\documentclass{amsart}

\usepackage[a4paper,top=3cm,bottom=2.5cm,
left=2.5cm,right=2.5cm,marginparwidth=60pt]{geometry}

\usepackage[english]{babel} 
\usepackage[utf8]{inputenc}
\usepackage[T1]{fontenc} 
\usepackage{braket} 
\usepackage{tikz}
\usetikzlibrary{math,shapes.callouts} 
\usepackage{tikz-3dplot}
\usepackage{tikz-cd} 

\tikzcdset{
	cells={font=\everymath\expandafter{\the\everymath\displaystyle}},
}

\usepackage{amsmath,amssymb} 

\usepackage[pdftex,colorlinks=true]{hyperref} 
\hypersetup{urlcolor=blue, citecolor=red, linkcolor=blue}

\newcommand{\linkProofSmoothability}{www.paololella.it/parity-conjecture/proof_smoothability.m2}
\newcommand{\linkProofCounterExample}{www.paololella.it/parity-conjecture/proof_counterexamples.m2}
\newcommand{\linkProofClosedImmersion}{www.paololella.it/parity-conjecture/proof_closed-immersion.m2}

\DeclareMathOperator{\Spec}{Spec} 
\DeclareMathOperator{\Mf}{Mf}
\DeclareMathOperator{\Hilb}{Hilb}
\DeclareMathOperator{\Gr}{Gr} 
\DeclareMathOperator{\colen}{colen}
\DeclareMathOperator{\len}{len}
\DeclareMathOperator{\Sym}{Sym}

\newcommand{\Calo}{\mathcal{O}} 
\newcommand{\Calq}{\mathcal{Q}} 
\newcommand{\Cals}{\mathcal{S}}
\newcommand{\C}{\mathbb{C}} 
\newcommand{\A}{\mathbb{A}}
\newcommand{\N}{\mathbb{N}}
\newcommand{\Pj}{\mathbb{P}}

\DeclareMathOperator{\Hom}{Hom}
\DeclareMathOperator{\Gl}{Gl}

\newtheorem{theorem}{Theorem}[section]
\newtheorem{proposition}[theorem]{Proposition}
\newtheorem{corollary}[theorem]{Corollary}
\newtheorem{lemma}[theorem]{Lemma}
\newtheorem{conjecture}{Conjecture}
\newtheorem*{theoremIntro}{Theorem}

\theoremstyle{definition}
\newtheorem{definition}[theorem]{Definition}

\theoremstyle{remark} 
\newtheorem{remark}[theorem]{Remark}
\newtheorem{example}[theorem]{Example}
\newtheorem*{conjecture*}{Conjecture} 

\numberwithin{equation}{section}

\title[A counterexample to the parity conjecture]{A counterexample to the parity conjecture}

\author[F. Giovenzana]{Franco Giovenzana}
\address[F. Giovenzana]{Laboratoire de Math\'ematiques d’Orsay\\ Universit\'e Paris-Saclay\\Rue Michel Magat, B\^at. 307, 91405\\ Orsay, France}
\email{\href{mailto:franco.giovenzana@universite-paris-saclay.fr}{franco.giovenzana@universite-paris-saclay.fr}}

\author[L. Giovenzana]{Luca Giovenzana}
\address[L. Giovenzana]{Fakult\"at f\"ur Mathematik\\ Technische Universit\"at Chemnitz\\
	Reichenhainer Stra\ss e 39, 09126 Chemnitz, Germany}
\curraddr{Department of Pure Mathematics\\ University of Sheffield\\ Hicks Building, Hounsfield Road\\ Sheffield, S3 7RH\\ UK}
\email{\href{mailto:l.giovenzana@sheffield.ac.uk}{l.giovenzana@sheffield.ac.uk}}

\author[M. Graffeo]{Michele Graffeo} 
\address[M. Graffeo]{Scuola Internazionale Superiore di Studi Avanzati (SISSA), Via Bonomea 265, 34136 Trieste, Italy}
\email{\href{mailto:mgraffeo@sissa.it}{mgraffeo@sissa.it}}

\author[P. Lella]{Paolo Lella} 
\address[P. Lella]{Dipartimento di Matematica\\ Politecnico di Milano\\ Piazza Leonardo da Vinci 32\\ 20133 Milan\\ Italy}
\email{\href{mailto:paolo.lella@polimi.it}{paolo.lella@polimi.it}}


\thanks{This work is partially supported by  the Thematic Research Programme \lq\lq Tensors: geometry, complexity and quantum entanglement\rq\rq, University of Warsaw, Excellence Initiative – Research University and the Simons Foundation Award No. 663281 granted to the Institute of Mathematics of the Polish Academy of Sciences for the years 2021-2023. F.G.'s research is funded by Deutsche Forschungsgemeinschaft (DFG, German Research Foundation), Projektnummer 509501007. L.G.~was supported by Engineering and Physical Sciences
	Research Council (EPSRC) New Investigator Award EP/V005545/1 \lq\lq Mirror
	Symmetry for Fibrations and Degenerations\rq\rq. L.G.~was supported by the DFG research grant Le 3093/3-2. M.G. and P.L. were partially supported by the project PRIN 2020 \lq\lq Squarefree Gr\"obner degenerations, special varieties and related topics\rq\rq~(MUR, project number 2020355B8Y).
	All authors are members of GNSAGA}

\begin{document}
	
\begin{abstract}
	Let $[Z]\in\Hilb^d\A^3$ be a zero-dimensional subscheme of the affine three-dimensional complex space of length $d>0$. Okounkov and Pandharipande have conjectured that the dimension of the tangent space to $\Hilb^d\A^3$ at $[Z]$ and $d$ have the same parity. The conjecture was proven by Maulik, Nekrasov, Okounkov and Pandharipande for points $[Z]$ defined by monomial ideals and very recently by Ramkumar and Sammartano for homogeneous ideals. In this paper we exhibit a family of zero-dimensional schemes in $\Hilb^{12} \A^3$, which disproves the conjecture in the general non-homogeneous case.
\end{abstract}
	
\maketitle

	\section{Introduction}\label{sec:intro}
	
	Given a quasi-projective variety $X$ defined over the field of complex numbers and a positive integer $d>0$, the Hilbert scheme $\Hilb^d X$ of $d$ points on $X$ is the scheme parametrising zero-dimensional subschemes of length $d$ of $X$. It is a quasi-projective scheme and it was introduced by Alexander  Grothendieck in \cite{FGA}.
	
	Over the last decades the study of Hilbert schemes of points has been a central topic of research and even though a number of results have been proven, several questions about their geometry remain open.
	By a classical result of Fogarty, we know that the Hilbert scheme $\Hilb^d X$ of a connected variety $X$ is connected for all $d$ \cite{FOGARTY}. If $X$ is  smooth and irreducible, then $\Hilb^d X$ is smooth, and hence irreducible, as long as $\dim X\leqslant 2$ \cite{FOGARTY}. In higher dimension, $\Hilb^d X$ is smooth for $\dim X\geqslant 3$ and $d\leqslant 3$, and singular otherwise. Recently, in \cite{joachim-pathologies} Jelisiejew showed that its singularities are pathological proving that Hilbert schemes of points satisfy Vakil’s Murphy’s Law up to retraction \cite{Vakil}. Another open question concerns the irreducibility of these schemes. It is known that, when $\dim X\geqslant 4$, $\Hilb^d X$ is irreducible for $d\leqslant 7$ and reducible otherwise \cite{Iarrob,MAZZOLA,8POINTS}. On the other hand, the problem of determining the irreducibility of $\Hilb^d X$ for a smooth irreducible threefold $X$ is only partially solved. Indeed, it has been recently proved that it is irreducible for $d\leqslant 11$  \cite{Klemen,10points,JOACHIM}, and by a classical result of Iarrobino it is reducible for $d\geqslant 78$ \cite{IARRO}, but nothing is known for the intermediate cases.
	
	It is worth concluding this historical introduction by mentioning that, although it is known that when $\dim X \geqslant 4$ and  $d$ is big enough,  $\Hilb^d X$ has non-reduced components, nothing is known about the reducedness of the Hilbert schemes of points on a smooth threefold \cite[Problem XV]{JoachimQuestions}, see also \cite{szachniewicz2021nonreducedness} for more examples.
	
	The study of the singularities of $\Hilb^d  X$ has benefited from renewed interest in the mathematical community due to the connection  with enumerative geometry established by Maulik, Nekrasov, Okounkov, and Pandharipande \cite{PANDHARIPANDE}.
	In particular, Okounkov and Pandharipande formulate the following conjecture
	
	\begin{conjecture}[Parity conjecture \cite{CONGETTURA}]\label{conjecture}
		Let $d\in\N$ be a positive integer and let $X$ be an irreducible smooth threefold. Then, for any $[Z]\in \Hilb^d X $, one has
		\[\dim_{\C} \mathsf T_{\mbox{\tiny $[Z]$}}\Hilb^d X\equiv d\  \pmod 2,\]
		where $ \mathsf T_{\mbox{\tiny $[Z]$}}\Hilb^d X$ denotes the tangent space of $\Hilb^d X$ at $[Z]$.
	\end{conjecture}
	Conjecture \ref{conjecture} was proven to hold for monomial ideals in \cite{PANDHARIPANDE} and very recently for homogeneous ideals in \cite{SAMMARTANO}.
	In this paper we exhibit counterexamples to Conjecture \ref{conjecture}, proving that the parity conjecture does not hold for $\Hilb^d \mathbb A^3$ whenever $d\geqslant 12$.
	For instance, the ideal
	\begin{equation}\label{eq:baby-example}
		\begin{aligned}
			I:=&(x+(y,z)^2)^2 + (y^3 - xz) \\
			=& (x^2, xy^2, xyz, xz^2, y^2z^2, yz^3, z^4, y^3-xz)
		\end{aligned}
	\end{equation}
	is immediately checked to be of colength 12 and standard routines on a computer algebra software like \textit{Macaulay2} \cite{M2} can compute that $\dim T_{[I]}\Hilb^{12} \A^3 =45$, hence showing the main result of the paper.
	\begin{theoremIntro}[Corollary~\ref{cor:conjecture}]
		The parity conjecture is false for any $d\geqslant 12$.
	\end{theoremIntro}
	
	We make here some comments. Our main result shows that the parity conjecture fails for $d\geqslant 12$ and at the moment we do not know any counterexample of smaller length. At the same time the irreducibility of $\Hilb^d \A^3$ is known for $d\leqslant 11$, while for $d=12$ it is not known whether the Hilbert scheme is irreducible.  From this observation one might wonder if irreducibility is related to the parity conjecture. To this direction, it is worth mentioning that the counterexample we present in Equation \eqref{eq:generic-counterexample} is smoothable as we show in Proposition \ref{prop:smoothability}. Indeed, it can be deformed to the disjoint union of 4 fat points as depicted in Figure \ref{fig:collision}.

	Further, even if it is possible to produce counterexamples of greater length by adding $k$-tuples of distinct points disjoint from the support of our example, we emphasize that we can exhibit other counterexamples, not obtained in this way. We give an example in Remark \ref{rem:other-cexamples}.
	
	The parity conjecture is also linked to a long-standing open problem in enumerative geometry, namely the constancy of the Behrend function $\nu_{\Hilb^d\A^3}$ on $\Hilb^d\A^3$ \cite{Beh,SignAnn}. Recently, in \cite{behrendnonconstant} it was shown how the failure of the parity conjecture implies the non-constancy of $\nu_{\Hilb^d\A^3}$. In a future project we will address the problem of computing the Behrend function at points disproving the parity conjecture by generalising the techniques introduced in \cite{GRICOLFI}.

\subsection*{Plan of the paper}
After recalling basic notions of Hilbert schemes and marked families in Section \ref{sec:preliminaries}, we prove the main result in Section \ref{sec:mainresults}. Also, we investigate the family of counterexamples given by ideals modeled on the one given in Equation \eqref{eq:baby-example}, and we show in Proposition \ref{prop:smoothability} that all of them lie in the smoothable component.
	
	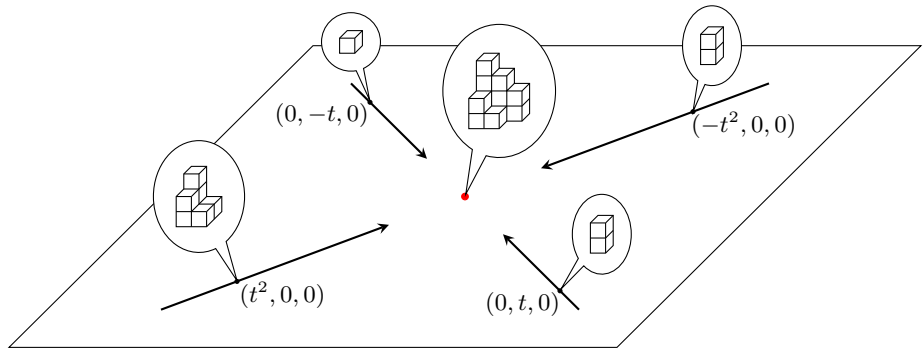
\begin{figure}
		\centering
		\begin{tikzpicture}
			\draw[] (-6,-2) -- (2,-2) -- (6,2) -- (-2,2) -- cycle;
			
			\draw [fill] (1.25,-1.25) circle (0.75pt);
			\node (A) at (1.25,-1.25) [xshift=-14pt,yshift=-5pt] {\small $(0,t,0)$};
			\draw[thick,-stealth] (1.5,-1.5) -- (0.5,-0.5);
			
			\draw [fill] (-1.25,1.25) circle (0.75pt);
			\node (B) at (-1.25,1.25) [xshift=-18pt,yshift=-5pt] {\small $(0,-t,0)$};
			\draw[thick,-stealth] (-1.5,1.5) -- (-0.5,0.5);
			
			\draw [fill] (-3,-1.125) circle (0.75pt);
			\node (C) at (-3,-1.125) [xshift=17pt,yshift=-5pt] {\small $(t^2,0,0)$};
			\draw[thick,-stealth] (-4,-1.5) -- (-1,-0.375);
			
			\draw [fill] (3,1.125) circle (0.75pt);
			\node (D) at (3,1.1666) [xshift=19pt,yshift=-6pt] {\small $(-t^2,0,0)$};
			\draw[thick,-stealth] (4,1.5) -- (1,0.375);

			\node[draw,ellipse callout, callout absolute pointer={(-1.255,1.28)},fill=white] at (-1.5,2.05) {\tikz[scale=0.2]{
					\draw (0,0) -- (1,0) -- (1.5,0.5) -- (1.5,1.5) -- (0.5,1.5) -- (0,1) -- cycle;
					\draw (1,1) -- (1,0);
					\draw (1,1) -- (0,1);
					\draw (1,1) -- (1.5,1.5);
			}};
			
			\node[draw,ellipse callout, callout absolute pointer={(1.265,-1.225)},fill=white] at (1.8,-0.5) {\tikz[scale=0.2]{
					\draw (0,0) -- (1,0) -- (1.5,0.5) -- (1.5,2.5) -- (0.5,2.5) -- (0,2) -- cycle;
					\draw (1,2) -- (1,0);
					\draw (1,2) -- (0,2);
					\draw (1,2) -- (1.5,2.5);
					\draw (1,1) -- (0,1);
					\draw (1,1) -- (1.5,1.5);
			}};
			
			\node[draw,ellipse callout, callout absolute pointer={(3,1.125)},fill=white] at (3.25,2) {\tikz[scale=0.2]{
					\draw (0,0) -- (1,0) -- (1.5,0.5) -- (1.5,2.5) -- (0.5,2.5) -- (0,2) -- cycle;
					\draw (1,2) -- (1,0);
					\draw (1,2) -- (0,2);
					\draw (1,2) -- (1.5,2.5);
					\draw (1,1) -- (0,1);
					\draw (1,1) -- (1.5,1.5);
			}};
			
			\node[draw,ellipse callout, callout absolute pointer={(-3,-1.115)},fill=white] at (-3.5,.) {\tikz[scale=0.2]{
					\draw (0,0) -- (2,0) -- (3,1) -- (3,2) -- (2,2) -- (2,4) -- (1,4) -- (0.5,3.5) -- (0.5,2.5) --(0,2) -- cycle;
					\draw (1,2) -- (1,0);
					\draw (1,2) -- (0,2);
					\draw (1,2) -- (2,3);
					\draw (2,1) -- (0,1);
					\draw (1,1) -- (2,2);
					\draw (2,1) -- (2,0);
					\draw (2,1) -- (3,2);
					\draw (2.5,1.5) -- (2.5,0.5);
					\draw (2.5,1.5) -- (1.5,1.5) -- (1.5,3.5) -- (2,4);
					\draw  (1.5,3.5) -- (0.5,3.5);
					\draw  (1.5,2.5) -- (0.5,2.5);
			}};
			
			\draw [fill=red,red] (0,0) circle (1.25pt);
			
			\node[draw,ellipse callout, callout absolute pointer={(0.01,0.03)},fill=white] at (.45,1.4) {\tikz[scale=0.2]{
					\draw (0,0) -- (2,0) -- (2.5,0.5) -- (3.5,0.5) -- (4,1) -- (4,3) -- (3,3) -- (3,4) -- (2,4) -- (2,5) -- (1,5) -- (0.5,4.5) -- (0.5,2.5) -- (0,2) -- cycle;
					\draw (0,1) -- (2,1) -- (2.5,1.5) -- (3.5,1.5) -- (4,2);
					\draw (0,2) -- (1,2) -- (1.5,2.5) -- (3.5,2.5) -- (4,3);
					\draw (0.5,2.5) -- (1.5,2.5) -- (1.5,4.5);
					\draw (0.5,3.5) -- (1.5,3.5) -- (2,4);
					\draw (0.5,4.5) -- (1.5,4.5) -- (2,5);
					\draw (1.5,3.5) -- (2.5,3.5) -- (3,4);
					\draw (2.5,3.5) -- (2.5,2.5) -- (3,3);
					\draw (1,2) -- (1,0);
					\draw (1.5,2.5) -- (1.5,1.5) -- (1,1);
					\draw (2.5,2.5) -- (2.5,0.5) ;
					\draw (3.5,2.5) -- (3.5,0.5) ;
					\draw (1.5,1.5) -- (2.5,1.5);
					\draw (2,1) -- (2,0);
			}};
		\end{tikzpicture}
		\caption{A degeneration of four fat points leading to a counterexample to the parity conjecture.}
		\label{fig:collision}
	\end{figure}

In Section \ref{sec:description} we describe the locus in $\Hilb^{12}\A^3$ that the family of counterexamples determine. We are not aware if these are all the ideals of length 12 for which the parity conjecture fails.
	
In Appendix \ref{APPENDICE} we provide and explain the computations necessary for the proof of the main results.

\subsection*{Acknowledgments}
We are grateful for the research conditions and the stimulating environment that we were provided with during the workshop \lq\lq Deformation Theory\rq\rq ~held in December 2022 in Warsaw as a part of the semester \textit{\lq\lq Algebraic Geometry with Applications to TEnsors and Secants\rq\rq}. Special thanks to Andrea Ricolfi for his support and guidance during the writing of the paper. It is our pleasure to thank Joachim Jelisiejew and Alessio Sammartano for useful discussions. Moreover, we thank Daniele Faenzi, Martijn Kool, and Reinier F. Schmiermann for the interest shown.

	\section{Preliminaries}\label{sec:preliminaries}
	In this section we recall some well known results and we settle some notation.
	
	Let $Z\hookrightarrow \A^n$ be a closed subscheme defined by the ideal $I_Z\subset\C[x_1,\ldots,x_n]$. Recall that the \textit{dimension} of $Z$ is defined as the Krull dimension (see \cite[Chapter II]{EISENBUD} for more details on dimension theory) of the ring $\Calo_Z=\C[x_1,\ldots,x_n]/I_Z$. When $\Calo_Z$ is zero-dimensional as a ring, then  it is a semilocal artinian $\C$-algebra of finite-type and, as a consequence, it is a finite dimensional vector space over the complex numbers. The complex dimension of $\Calo_Z$ is called the \textit{length} of $Z$ or the \textit{colength} of $I_Z$,
	\[
	\len Z=\colen I_Z = \dim_\C \Calo_Z.
	\]
	When $Z$ is a zero-dimensional closed subscheme of $\A^n$ and $\Calo_Z$ is a local $\C$-algebra we will say that $Z$ is a \textit{fat point}.
	
	Let $d$ be a positive integer, and let $X$ be a smooth quasi-projective variety. Recall that the \textit{Hilbert functor} of $d$ points in $X$ is the association $\underline{\Hilb}^d X: \textnormal{Schemes}^{op} \to \textnormal{Sets}$ defined by
	\[
	\left(\underline{\Hilb}^d X\right)(S) = \Set{ \mathcal Z \hookrightarrow X\times S \mbox{ closed subscheme} | \mathcal Z \to S \mbox{ is flat   and finite of degree }d}.
	\]
	
	By a celebrated result of Grothendieck, the functor $\underline{\Hilb}^d X$ is representable  and the fine moduli space $\Hilb^dX$ representing it is a quasi-projective scheme called \textit{Hilbert scheme} \cite{FGA}. As there is a bijection between closed subschemes $Z$ and their ideal sheaves $I_Z$, we will denote points of the Hilbert scheme by $[Z]$ or $[I_Z]$.
	
	Although it is not known, in general, if the Hilbert scheme $\Hilb^d \A^3$ is irreducible, there is a component which can always be defined. Precisely, the \emph{smoothable component} is defined as the closure of the open subscheme $U\subset\Hilb^d\A^3$ parametrising closed and reduced zero-dimensional subschemes of length $d$ of $\A^3$.
	
	\begin{definition}
		A point $[Z]\in \Hilb^d\A^3$ is \textit{smoothable} if it belongs to the smoothable component.
	\end{definition}

	We report now the description of the tangent space of $\Hilb^dX$ in terms of first order deformations. Let $\C[\epsilon] = \C[t]/(t^2)$ be the ring of \textit{dual numbers}, where $\epsilon$ stands for the equivalence class of $t$ in the quotient ring, and let $D = \Spec\, \C[\epsilon]$ be its spectrum. A \emph{first order deformation} of a scheme $Z\subset \A^n$ is a commutative diagram
	\begin{center}
		\begin{tikzpicture}[xscale=2,yscale=1.2]
			
			\node (A) at (0,0) [] {$Z$};
			\node (B) at (1,0) [] {$\mathcal{D}$};
			\node (C) at (0,-1) [] {$\Spec \C$};
			\node (D) at (1,-1) [] {$D$};
			\draw [->] (A) -- (B);
			\draw [->] (A) -- (C);
			\draw [->] (B) --node[right]{\tiny $f$} (D);
			\draw [right hook->] (C) -- (D);
		\end{tikzpicture}
	\end{center}
	where $f$ is a flat morphism, such that the induced morphism $Z \to \mathcal{D} \times_{D} \Spec \C$ is an isomorphism. As a consequence of the universal property of the Hilbert scheme, we can associate to each first order deformation of $Z$ a morphisms of schemes $D \to \Hilb^d \A^n$ mapping the unique closed point of $D$ to $[I_Z]$. We denote the collection of these morphisms by $\textnormal{Hom}_{I_Z}(D, \Hilb^d \A^n)$.
	Recall that this set has a canonical structure of vector space and it is canonically isomorphic to the tangent space $\mathsf{T}_{[I_Z]} \Hilb^d\A^n$.
	The following result gives a different characterisation of the tangent space of the Hilbert scheme at a given point $[I]\in \Hilb^d\A^n$.
	
	\begin{theorem}[{\cite[Corollary~6.4.10]{FGAexplained}}]
		Let $[I]\in \Hilb^d\A^n$ be any point and let $\mathsf{T}_{[I]}\Hilb^{d}\A^n$ denote the tangent space of $  \Hilb^{d}\A^n$ at $[I]$. Then,
		\[ 
		\mathsf T_{[I]}\Hilb^{d}\A^n \simeq \Hom_{\C[x_1,\ldots,x_n]}(I,\C[x_1,\ldots,x_n]/I) \simeq \textnormal{Hom}_{I}(D, \Hilb^d \A^n).\]
	\end{theorem}
	
	In Section \ref{sec:mainresults} we give an explicit description of the tangent space $\mathsf T_{[I]}\Hilb^{d}\A^n$ for an ideal $I$ disproving the parity conjecture. Our main tools are the \textit{marked bases}\footnote{The theory of marked bases has been introduced for homogeneous ideals and the projective setting. The affine version we refer to has been developed in \cite{BertoneCioffiRoggero}. For the zero-dimensional case, see also \cite{BertoneCioffi}.} (see \cite{CioffiRoggero,BertoneCioffiLellaRoggero,BertoneLellaRoggero,LellaRoggero,CeriaMoraRoggero,BertoneCioffiRoggero} and references therein). We recall here the main definitions and properties and we postpone more details to Appendix \ref{APPENDICE}.
	
	In this paper, we describe monomials in $\C[x_1,\ldots,x_n]$ via the standard multi-index notation. Namely, for any $\underline{\alpha} =(\alpha_1,\ldots,\alpha_n)\in \mathbb{Z}_{\geqslant 0}^{n}$, the symbol $x^{\underline{\alpha}}$ stands for $x_1^{\alpha_1} \cdots x_n^{\alpha_n}$. 
	
	A set of monomials $\mathcal{N}$ in $\C[x_1,\ldots,x_n]$ is called \emph{order ideal} if it is closed under subdivision, i.e.
	\[
	x^{\underline{\alpha}} \in \mathcal{N}\text{~and~}x^{\underline{\beta}} \mid x^{\underline{\alpha}}\quad\Rightarrow\quad x^{\underline{\beta}} \in \mathcal{N}.
	\]
	The monomials in the complementary set of $\mathcal{N}$ generate a monomial ideal in $\C[x_1,\ldots,x_n]$ that we denote by $J_{\mathcal{N}}$. Notice that $J_{\mathcal{N}}\cap \mathcal{N}=\emptyset$. 
	
	If the order ideal is finite, then the Krull dimension of the quotient ring $\C[x_1,\ldots,x_n]/J_{\mathcal{N}}$ is 0 and  $\mathcal{N}$ is the unique monomial basis of $\C[x_1,\ldots,x_n]/J_{\mathcal{N}}$ as $\C$-vector space, that is $J_{\mathcal{N}}$ defines a zero-dimensional scheme in $\A^n$ of length $\vert \mathcal{N}\vert$.
	
	\begin{definition}[{\cite[Definition 6.2]{BertoneCioffiRoggero}}]
		Let $\mathcal{N} \subset \C[x_1,\ldots,x_n]$ be a finite order ideal. The marked family functor associated to $\mathcal{N}$ is the covariant functor $\underline{\Mf}_\mathcal{N} : \C\text{-Algebras} \to \text{Sets}$ defined by
		\begin{equation}
			\underline{\Mf}_{\mathcal{N}}(A) = \left\{ I \subset A[x_1,\ldots,x_n]\ \middle\vert\ A[x_1,\ldots,x_n] = I \oplus \langle \mathcal{N}\rangle\right\}. 
		\end{equation}
	\end{definition}
	
	The marked family functor $\underline{\Mf}_{\mathcal{N}}$ is a subfunctor of the Hilbert functor, since for every ideal $I\in\underline{\Mf}_{\mathcal{N}}(A)$, the quotient algebra $A[x_1,\ldots,x_n]/I$ turns out to be a free $A$-module with a monomial basis given by the order ideal $\mathcal{N}$. Moreover, the functor $\underline{\Mf}_{\mathcal{N}}$ is an open representable subfunctor of the Hilbert functor \cite[Theorem 6.6, Proposition 6.13]{BertoneCioffiRoggero}. Hence, the fine moduli space $\Mf_{\mathcal{N}}$ representing this functor is an open subscheme of $\Hilb^{\vert\mathcal{N} \vert}\A^n$.

	The collection of schemes 
	\[
	\Set{\Mf_{\mathcal{N}}|\mathcal N\subset \C[x_1,\ldots,x_n] \mbox{ is an order ideal of cardinality }d}
	\]
	is an atlas of $\Hilb^d \A^n$ \cite[Proposition 5]{BertoneCioffi}. Thus, given a point $[I] \in \Hilb^d \A^n$, we can always compute $\mathsf T_{[I]}\Hilb^{d}\A^n$ as  $\mathsf T_{[I]}\Mf_{\mathcal{N}} \simeq \textnormal{Hom}_{I}(D, \Mf_{\mathcal{N}})$ for a suitable order ideal $\mathcal{N}$.

	Since we work with closed zero-dimensional subschemes of $\A^3$, we fix once for all the ring $R$ to be the polynomial ring in three variables and complex coefficients $R=\C[x,y,z]$ and the ideal $\mathfrak m\subset R$ to be the unique maximal homogeneous ideal $\mathfrak m=(x,y,z)$. Moreover, we denote the $k$-th graded piece of a graded module $M$ by $M_k$. For instance, we have
	\[R_k=\Set{p\in \C[x,y,z] | p \mbox{ is homogeneous of degree } \deg p=k  }\cup \Set{0}. \]
	
	We conclude this section by warning that, with abuse of notation, we denote both the elements of $R$ and the elements of $R/I$ in the same way, tacitly assuming that we are choosing representatives.
	
	\section{Main results}\label{sec:mainresults}

	In the following, given a subspace $W$ of a vector space $V$ we refer to any subspace $U$ of $V$ such that $V=W \oplus U$ as a complement of $W$.
	
	\begin{theorem}\label{thmA}
		Let $\ell\in R_1\smallsetminus \Set{0}$ be a nonzero linear form and let $I\subset R $ be an ideal of the form \begin{equation}\label{eq:generic-counterexample}
			I = ((\ell)+\mathfrak m^2)^2 +(v)  ,
		\end{equation}
		where, $v\in (\Sym^3 (L) +\ell \cdot L)\smallsetminus \Set{0}$ for some $L\in\Gr(2,R_1)$ complement of $\langle \ell \rangle$. Then, $I$ has colength $\colen I= 12$ and, for a general choice of $v$, we have
		\[ 
		\dim_\C \mathsf T_{[I]}\Hilb^{12}\A^3 = 45.
		\]
	\end{theorem}
	
	\begin{proof}
		Without loss of generality, we can suppose $\ell=x$ and $L=\langle y,z\rangle $.
		Under these assumptions, the ideal $I$ takes the following form
		\[
		I=((x)+(y,z)^2)^2+\big(b_0y^3+b_1y^2z+b_2yz^2+b_3z^3+b_4xy+b_5xz\big),
		\]
		for some $[b_0:\cdots:b_5]\in\Pj^5$.
		
		In order to prove the first part of the statement,  first notice that the ideal 
		\[J=((x)+(y,z)^2)^2 \]
		has colength $\colen J=13$. Now, we also have
		\[ b_0y^3+b_1y^2z+b_2yz^2+b_3z^3+b_4xy+b_5xz\notin J\]
		which implies $\colen I=12$ for all $I$ of the form \eqref{eq:generic-counterexample}.
		
		We move now to the computation of the tangent space $ \mathsf T_{[I]}\Hilb^{12}\A^3$ for a general $I$ of the form \eqref{eq:generic-counterexample}. By generality of $v$, we set $b_0=1$ and the ideal $I$ is minimally generated by
		\[
		x^2,\quad  xy^2, \quad xyz,\quad xz^2,\quad y^2z^2,\quad yz^3,\quad z^4, \quad v = y^3 + b_1y^2z+b_2yz^2+b_3z^3+b_4xy+b_5xz.
		\]
		Thus, the quotient algebra $\C[x,y,z]/I$ admits the monomial basis
		\[
		\mathcal{N} = \left\{ y^2z,yz^2,z^3,xy, y^2, xz, yz, z^2, x, y, z, 1\right\}
		\]
		and $[I] \in \Mf_{\mathcal{N}}$.
		We can compute the tangent space $\mathsf T_{[I]}\Hilb^{12}\A^3$ as $\mathsf T_{[I]}\Mf_{\mathcal{N}} \simeq \textnormal{Hom}_I(D,\Mf_{\mathcal{N}})$. If the polynomial $\mathfrak B:= b_{3}b_{4}^{3}-b_{2}b_{4}^{2}b_{5}+b_{1}b_{4}b_{5}^{2}-b_{5}^{3}$ is different from zero, the first order deformations are described by the ideal generated by the following polynomials
		\begin{equation}\label{eq:TS}
			\begin{split}
				x^2 &{} + \epsilon_{1}\, y^2z+ \epsilon_{2}\, y z^2+ \epsilon_{3}\, z^3 + \epsilon_{4}\, x y+ \epsilon_{5}\, x z+2  \epsilon_{10} x \\
				x y^2&{} + \epsilon_{6}\, y^2z+ \epsilon_{7}\, y z^2+ \epsilon_{8}\, z^3+ \epsilon_{9} \, x y+ \epsilon_{10}\, y^2+ \epsilon_{11}\, x z \\
				x y z &{} + \epsilon_{12}\, y^2z+ \epsilon_{13}\, y z^2+ \epsilon_{14}\, z^3+ \epsilon_{15}\, x y+ \epsilon_{16}\, x z+ \epsilon_{10}\, y z\\
				x z^2 &{} + \epsilon_{17}\, y^2z+ \epsilon_{18}\, y z^2+ \epsilon_{19}\, z^3+ \epsilon_{20}\, x y+ \epsilon_{21}\, x z+ \epsilon_{10}\, z^2\\
				y^2z^2 &{} + \epsilon_{22}\, y^2z+ \epsilon_{23}\, y z^2+ \epsilon_{24}\, z^3+ \epsilon_{25}\, x y+ \epsilon_{26}\, x z\\
				y z^3 &{} + \epsilon_{27}\, y^2z+ \epsilon_{28}\, y z^2+ \epsilon_{29}\, z^3+ \epsilon_{30}\, x y+ \epsilon_{31}\, x z\\
				z^4&{} + \epsilon_{32}\, y^2z+ \epsilon_{33}\, y z^2+ \epsilon_{34}\, z^3+ \epsilon_{35}\, x y+ \epsilon_{36}\, x z\\
				v &{} + \epsilon_{37}\, y^2z + \epsilon_{38}\, y z^2 + \epsilon_{39}\, z^3+ \epsilon_{40}\, x y+ \epsilon_{41}\, y^2+ \epsilon_{42}\, x z+ \epsilon_{43}\, y z+ \epsilon_{44}\, z^2 + {} \\[-4pt] & {} + \epsilon_{45}\, x+b_{4} \epsilon_{10}\, y+b_{5} \epsilon_{10}\, z
			\end{split}
		\end{equation}
		in $\C[\epsilon_1,\ldots,\epsilon_{45}][x,y,z]$, where $\epsilon_i$ is the equivalence class of $t_i$ in $\C[t_1,\ldots,t_{45}]/(t_1,\ldots,t_{45})^2$ (see Appendix \ref{subsec:counterexample} for the computational details)\footnote{Notice that in the appendix the polynomial $v$ is denoted by $f_{y^3}$.}. Thus, we have
		\[\dim_\C \mathsf T_{[I]}\Hilb^{12}\A^3 = 45.\qedhere\]
	\end{proof}
	\begin{corollary}\label{cor:conjecture}
		The parity conjecture (Conjecture \ref{conjecture}) is false for any $d\geqslant 12$.
	\end{corollary}
	\begin{proof}
		For $d=12$ a counterexample is provided by Theorem \ref{thmA}. The case $d= 12+k$ with $k\geqslant 1$ is treated by adding a $k$-tuple of distinct points disjoint from the support of the counterexample of length 12.
	\end{proof}
 
    \begin{remark}[{cf.~\cite[Remark 11, Example 12]{SAMMARTANO}}] In \cite{SAMMARTANO}, the authors prove that the parity conjecture holds for any ideal homogeneous with respect to a grading taking values in a torsion-free abelian
group such that $\deg(x) + \deg(y) + \deg(z)$ is not divisible by 2. Our example \eqref{eq:baby-example} is in fact homogeneous with respect to the grading $\deg:R \to \mathbb{Z}^2$ defined by $\deg(x) = (3,0)$, $\deg(y) = (1,1)$ and $\deg(z) = (0,3)$. However, $\deg(x) + \deg(y) + \deg(z) = 2(2,2)$.
    \end{remark}
    
	\begin{remark}\label{rem:other-cexamples}
		It is worth mentioning that with the help of the computer we are able to produce many counterexamples apart from the one in the proof of Corollary \ref{cor:conjecture}. For instance, the following is a counterexample of length 78 and tangent space of complex dimension 263.
		\[
		\begin{aligned}
			I=(&y^2z^4,x^2z^4+z^6,y^3z^3,x^3z^3,xy^3z^2,xy^4z+x^3yz^2-y^2z^3,x^4yz-z^6+y^3z^2,\\
   &x^3y^3-x^2yz^2-yz^4, x^5z+x^3y^2z+xyz^4-x^2z^3-z^5,y^6+x^4z^2+xy^2z^3+xz^5,\\
   &x^6-y^4z-xz^4,y^5z^2+xyz^5,x^2y^5+xyz^5+z^7,\\
   &x^5y^2+x^3y^2z^2+xyz^5+x^2y^4-y^4z^2+z^6-xy^2z^2).
		\end{aligned}
		\]
		In particular it is not of the form described in the proof of Corollary \ref{cor:conjecture}.
	\end{remark}  
	\begin{corollary}\label{cor:afffamily}
		The ideal 
		\[
		\mathcal I_x =((x+a_1 y+a_2 z)+(y,z)^2)^2+\big(b_0y^3+b_1y^2z+b_2yz^2+b_3z^3+b_4xy+b_5xz\big)\subset R[a_1,a_2][b_0,\ldots,b_5]
		\]
		defines a family $\mathcal Z_{\mathcal{I}_x} \subset \A^3\times \A^2\times \Pj^5$ of zero-dimensional subschemes of $\A^3$ flat over $\A^2\times \Pj^5$.
		
		Analogously, there are ideals $\mathcal{I}_y$ and $\mathcal{I}_z$ defining families $\mathcal Z_{\mathcal{I}_y},\mathcal Z_{\mathcal{I}_z} \subset \A^3\times \A^2\times \Pj^5$ of zero-dimensional subschemes of $\A^3$ flat over $\A^2\times \Pj^5$.
	\end{corollary}
	We move now to the study of the collection of ideals of the form \eqref{eq:generic-counterexample}.
	
	\begin{lemma}\label{lemmatech}
		Let us denote by $\Lambda_{\ell,L}$ the following set of ideals
		\[\Lambda_{\ell,L}=\Set{((\ell)+\mathfrak m^2)^2 +(v)|v\in(\Sym^3 (L) +\ell \cdot L)\smallsetminus\Set{0}}, \]
		where  $\ell\in R_1\smallsetminus \Set{0}$ is a nonzero linear form and $L\in\Gr(2,R_1)$ is a complement of $\langle\ell \rangle $.   Then, the set $\Lambda_{\ell,L}$ does not depend on the choice of $L$, i.e.  for any two complements $L,L'\in \Gr(2,R_1)$ of $\langle \ell\rangle $ we have 
		\[\Lambda_{\ell,L}=\Lambda_{\ell,L'}.\]
		We denote this set by $\Lambda_\ell$.
	\end{lemma}
	\begin{proof}
		Let $\ell\in R_1\smallsetminus \Set{0}$ be a nonzero linear form and let  $L,L'\in\Gr(2,R_1)$ be two complements of $\langle \ell\rangle$. Then, there are $\ell_1,\ell_2\in R_1$ linear forms and $(\alpha,\beta)\in\C^2$ such that
		\[\Set{\ell_1,\ell_2} \mbox{ and } \Set{\ell_1'=\ell_1+\alpha\ell,\ \ell_2'=\ell_2+\beta\ell} \]
		are bases of $L$ and $L'$ respectively.
		
		Now, the two ideals
		\begin{equation}\label{eq:LL'}
			\begin{matrix}
				\big((\ell) + \mathfrak m^2\big)^2 + \big(b_0\ell_1^3+b_1\ell_1^2\ell_2+b_2\ell_1\ell_2^2+b_3\ell_2^3+b_4\ell\ell_1+b_5\ell\ell_2\big)\\[1ex]
				\big((\ell) + \mathfrak m^2\big)^2 + \big(b_0\ell_1'^3+b_1\ell_1'^2\ell_2'+b_2\ell_1'\ell_2'^2+b_3\ell_2'^3+b_4\ell\ell_1'+b_5\ell\ell_2'\big)
			\end{matrix}
		\end{equation}
		are the same for any choice of $(\alpha,\beta)\in \C^2$. And this completes the proof of the lemma.
	\end{proof}
	\begin{remark}\label{rem:LAMBDAell}
		Let $\ell\in R_1\smallsetminus \Set{0}$ be a nonzero linear form. Let us consider $\Delta_\ell $  defined as follows
		\[\Delta_{\ell}= \Pj (\Sym^3 (R_1/\langle \ell\rangle) \oplus (\langle\ell\rangle \underset{\C}{\otimes} (R_1/\langle \ell\rangle)))\cong\Pj^5. \]
		Now, given a complement $L$ of $\langle\ell\rangle$, the restriction of the natural projection $\pi_\ell:R_1\rightarrow R_1/\langle \ell\rangle$ to $L$ induces a bijection $\gamma_\ell: \Delta_{\ell}\to \Lambda_{\ell}$. Precisely, we have
		\[
		\begin{tikzcd}[row sep=tiny]
			\Delta_\ell \arrow{r}{\gamma_\ell}&\Lambda_\ell\\ 
			{[}(p,\ell \otimes q){]} \arrow[mapsto]{r}&\left(\ell+\mathfrak m^2\right)^2+\left(\overline{p}+\ell \overline{q}\right)
		\end{tikzcd}
		\]
		where $\overline{p}=(\Sym\pi_{\ell})^{-1}(p)\cap \Sym (L)$ and $ \overline{q}=\pi_\ell^{-1}(q)\cap L$.
		Equation \eqref{eq:LL'} shows that this map does not depend on the choice of the complement $L$.
		
		We abuse notation and we interpret an element in $\Delta_{\ell}$ as the corresponding ideal in $\Lambda_{\ell}$.
		As a consequence of Theorem \ref{thmA}, the map $\gamma_\ell$ introduced in Remark \ref{rem:LAMBDAell} induces a morphism $\Delta_\ell\rightarrow \Hilb^{12}\A^3$ that, with abuse of notation, we denote by $\gamma_\ell$ as well.
		The image of $\gamma_\ell$ coincides with the image of $\{(a_1,a_2)\}\times \Pj^5$ under the classifying morphism induced by the family $\mathcal Z_{\mathcal{I}_x}$ from Corollary \ref{cor:afffamily}.
	\end{remark}
	
	We conclude this section showing that the ideals of the form \eqref{eq:generic-counterexample} are smoothable.
	
	\begin{proposition}\label{prop:smoothability}
		Any ideal $I\subset R$ of the form \eqref{eq:generic-counterexample} is smoothable.
	\end{proposition}
	\begin{proof}
		Since being smoothable is a closed condition, it is enough to prove the statement for the general $I\subset R$ of the form \eqref{eq:generic-counterexample}.
		Without loss of generality, up to the $\Gl(3,\C)$-action, we can suppose $\ell=x$. Then, $I$ takes the following form 
		\[
		I=\big((x)+(y,z)^2\big)^2+(y^3 + b_1 y^2 z + b_2 y z^2 + b_3z^3 + b_4 xy + b_5 xz).
		\]
		In order to prove the statement, it is enough to exhibit a flat family $\mathcal Z\hookrightarrow \A^1\times \A^3$ of zero-dimensional subschemes of $\A^3$ such that the fiber over the origin $\mathcal Z_0$ is $I$ and such that the general fiber $\mathcal Z_t$ is supported at least two points and hence it is smoothable (see \cite{JOACHIM}).
		
		For every counterexample $I\subset R$ to the parity conjecture discussed in Theorem \ref{thmA}, the tangent vector
		\[
		\left(x^2,\ xy^2,\ xyz + \boldsymbol{\epsilon y^2z},\ xz^2,\ y^2z^2,\ yz^3,\ z^4,\  y^3 + b_1y^2z+b_2yz^2+b_3z^3+b_4xy+b_5xz \right)
		\]
		lifts to the deformation given by the following polynomials
		\begin{equation}\label{eq:smoothingDef}
			\begin{split}
				x^{2}&{} \frac{b_4^2(b_3 b_4 - b_2b_5)}{\mathfrak{B}}t^2\, y^2 - \frac{b_4 (b_3b_4-b_2b_5)(b_1b_4-b_5)}{\mathfrak B}t^2\, yz + \frac{b_3b_4b_5(b_1b_4-b_5)}{\mathfrak B}t^2\, z^2 \\
				&{}- \frac{b_4\big(2\mathfrak{B} - b_5^2(b_1b_4-b_5)\big)}{\mathfrak B}t^2\, x + \frac{b_4^4(b_3b_4-b_2b_5)}{\mathfrak B}t^4,\quad
				xy^2,\quad xyz+t\,y^{2}z-b_4t^{2}\, yz,\\ xz^{2}&{}-b_4t^{2}\,z^{2},\quad y^{2}z^{2},\quad  yz^{3},\quad z^{4},\quad y^{3}+b_1y^{2}z+b_2yz^{2}+b_3z^{3}+b_4xy+b_5xz-b_4^{2}t^{2}\ y-b_4b_5t^{2}\, z
			\end{split}
		\end{equation}
		whenever the polynomial $\mathfrak{B} 
		= b_{3}b_{4}^{3}-b_{2}b_{4}^{2}b_{5}+b_{1}b_{4}b_{5}^{2}-b_{5}^{3} = b_4^2(b_3b_4 - b_2b_5) + b_5^2(b_1b_4-b_5)$ is different from zero.
		Notice that away from the zero locus of the polynomial $\mathfrak B$ the ideal $I$ has tangent space of dimension 45 as proven in Theorem \ref{thmA}.
		If $b_4 \neq 0$, the generic fiber is the disjoint union of four schemes of length 7, 2, 2 and 1 (see Appendix \ref{subsec:smoothability} for the details).
	\end{proof}

\subsection{Comments} 
In this subsection we make some comments regarding our main result.

The first observation is that Theorem \ref{thmA} and Proposition \ref{prop:smoothability} admit more direct proofs that does not require the use of the computer. We decided to present them in this form because, as an outcome, we obtain a more accurate description of the locus parametrising the counterexamples.
Precisely, in Theorem \ref{thmA} we obtain the polynomial ${\mathfrak B}$ which defines the complement of this locus. In Proposition \ref{prop:smoothability} instead, we obtain an explicit deformation of our counterexample into a scheme with support consisting of four distinct points.
It is worth mentioning that we found the first counterexamples via \textit{Macaulay2} computations using marked basis as main tool. Also for this reason we have presented the proofs in this form.

The Hilbert scheme of 12 points on a smooth threefold is interesting for several reasons. For instance 12 is the minimal number of points for which the irreducibility of the Hilbert scheme is not known.
It is then natural to ask whether the failure of parity conjecture is peculiar of $\Hilb^{d}\A^3$ for $d\geqslant 12 $ or if it fails already for $\Hilb^{11}\A^3$. Our method does not produce any counterexample of length smaller than 12. Nevertheless, we tested  a large number of ideals of colength 11 without finding new counterexamples.
Precisely, starting from the description of the so called Hilbert--Samuel stratification of the punctual Hilbert scheme of 11 points described in \cite{JOACHIM} and using the marked basis technology, we have generated some examples of non-homogeneous ideals in each of the Hilbert--Samuel strata so certifying the veridicity of the conjecture at each of these points.
A complete proof of the parity conjecture for the  Hilbert scheme of 11 points is outside the scope of this article. Hence, we leave this question open.

We conclude this section with a remark regarding the characteristic of the base field. A direct check shows that our counterexamples are valid in any characteristic other than 2 (see Remark \ref{rem:dim char}). We have tested in characteristic 2 all the counterexamples we found (see Remark \ref{rem:other-cexamples}) and none of them disproves the conjecture. Therefore, the parity conjecture remains open in characteristic 2.

	\section{Explicit description of the locus of counterexamples}\label{sec:description}
	
	We describe now the locus in $\Hilb^{12} \A^3$ of the counterexamples to the parity conjecture given in Theorem \ref{thmA}. It is worth pointing out that we do not know if the locus described comprise all the counterexamples of length less or equal than 12 to Conjecture \ref{conjecture}.
	
	We consider the projective space $\Pj(R_1)$ parametrising 1-dimensional linear subspaces of $R_1$ and the tautological exact sequence on it
	\[
	\begin{tikzcd}
		0\arrow{r}&
		\Cals\arrow{r}&
		\Calo_{\Pj(R_1)}\otimes R_1\arrow{r}&
		\Calq\arrow{r}&0,
	\end{tikzcd}
	\]
	where $\Cals$ is the universal subbundle, and $\Calq$ is the universal quotient bundle. We set $E$ to be the vector bundle
	\[ 
	E=\Sym^3 (\Calq)\oplus (\Cals\underset{\Calo_{\Pj(R_1)}}{\otimes}\Calq)
	\]
	and we consider the associated projective bundle $\Pj(E)$ over $\Pj(R_1)$. Its fiber $\Pj(E)_\ell$ over the point $[\ell]\in\Pj(R_1)$ is the space $\Delta_\ell$ introduced in Remark \ref{rem:LAMBDAell}.
	The morphism defined by the association 
	\[
	\begin{tikzcd}[row sep=tiny]
		\Pj(E) \arrow{r}& \Hilb^{12} \mathbb A^3\\
		({[}\ell{]},{[}v{]})\arrow[mapsto]{r}& \gamma_\ell\left({[}v{]}\right), 
	\end{tikzcd}
	\]
	where $\gamma_\ell$ is the bijection defined in Remark \ref{rem:LAMBDAell}, induces a family $\mathcal Z_E\subset \Pj(E)\times \A^3$ of zero-dimensional schemes of length 12.	
	
	Consider the coordinate atlas on $\Pj(E)$ induced by the basis $\Set{x,y,z}$ of $R_1$, then the restrictions of the family to the coordinate charts agree with the families $\mathcal Z_{\mathcal{I}_x},\mathcal Z_{\mathcal{I}_y},\mathcal Z_{\mathcal{I}_z} $ in Corollary \ref{cor:afffamily}.
	
	Finally, we enlarge our family by acting via translations and we get a new family $\mathcal Z\subset \Pj(E)\times \A^3\times \A^3$ of fat points non-necessarily supported at the origin $0\in \A^3$.
	
	In what follows we shall denote by $V$ the product $\Pj(E)\times \A^3$ and by $\theta: V\to \Hilb^{12}\A^3$ the classifying morphism.
	
	\begin{remark}
		Notice that set-theoretically the image $\theta(V)$ consists of the ideals described in
		\[
		\theta(V)=\Set{ [I]\in \Hilb^{12} \A^3 | I \mbox{ is of the form \eqref{eq:generic-counterexample} } }.
		\]
	\end{remark}
	
	\begin{theorem}
		The morphism $\theta\colon V\to\Hilb^{12}\A^3$ described above is a closed immersion of the smooth and irreducible 10-dimensional variety $V$.
	\end{theorem}
	
	\begin{proof}
		We start proving that the map $\theta$ is universally closed.
		Let $\Hilb_0^{12}\A^3$ be the projective subscheme of $\Hilb^{12}\A^3$ parametrising fat points $Z\subset \A^3$ of length 12 supported at the origin $0\in\A^3$. Let also $ \A^3\times \Hilb_0^{12}\A^3\subset \Hilb^{12}\A^3$ be the closed subscheme of $\Hilb^{12}\A^3$ parametrising all the fat points $Z\subset \A^3$ of length 12. The map $\theta|_{\Pj(E)\times \Set{0}}$ takes values in $\Hilb_0^{12}\A^3$ and it is proper because $\Pj(E)$ is projective. Finally, the map $\theta$ is universally closed because the square in the following diagram 
		\[
		\begin{tikzcd}
			V \arrow{r}\arrow{d}\arrow[bend left=20]{rr}{\theta} &\A^3\times \Hilb_0^{12}\A^3 \arrow{r}\arrow{d} &  \Hilb^{12}\A^3\\
			\Pj(E) \arrow{r} & \Hilb_0^{12}\A^3
		\end{tikzcd}
		\]
		is cartesian.
		
		The map $\theta$ is injective because different points in $V$ lead to different ideals in $ \Hilb^{12} \A^3$ as it can be seen as follows. It is enough to treat fat points supported at the origin, i.e. points of $\Pj(E)\times \Set{0}$ which we identify with $\Pj(E)$. Now to conclude, the bijection $\gamma_\ell$ maps different points of $\Delta_\ell$ to different ideals, and any ideal $\gamma_\ell(\Delta_\ell)$ contains up to scalar the square of a unique linear form, namely $\ell$.
		
		In order to prove the statement, we show that the differential $d\theta$ is full rank at every point.
		It is enough to prove the claim for ideals belonging to the open $U_x\subset \Pj(E)\times \{0\}\subset V$ consisting of ideals of the form
		\[
		\big((x + a_1 y + a_2 z\big) + \big(y,z)^2 \big)^2 + \big(b_0 y^3 + b_1 y^2z + b_2 yz^2 + b_3 z^3 + b_4 xy + b_5 xz \big).
		\]
		
		The proof now consists of a direct and standard computation, but quite lengthy to be reported on paper. Hence, we exhibit the result only for the ideals lying in the open subset defined by the condition $b_0\not=0$. 
		We refer to the \emph{M2} ancillary file \href{\linkProofClosedImmersion}{\tt proof\_closed-immersion.M2} for the complete discussion.
		
		Concretely, given $\ell  = x + a_1 y + a_2 z$, and  $v = y^3 + b_1y^2z + b_2 yz^2 + b_3 z^3 + b_4 xy + b_5 xz,$  we show that the differential
		\[
		d\theta_{\left([\ell],[v],0\right)}: \mathsf{T}_{[\ell]} \Pj(R_1) \times \mathsf{T}_{[v]} \Delta_\ell \times \mathsf{T}_0 \mathbb{A}^3 \to \mathsf{T}_{\gamma_\ell ([v])} \Hilb^{12} \mathbb{A}^3
		\]
		is injective.
		
		\smallskip
		
		Let us denote by $\{\partial_{\tau_x},\partial_{\tau_y},\partial_{\tau_z}\}\subset \mathsf{T}_0 \mathbb{A}^3 $ be the basis corresponding to the variables $x,y,z$. The image of $d\theta$  at the point $\left([\ell],[v],0\right)$  is generated by the first order deformations
	
		\[
		\begin{split}
			\mathsf{T}_{[\ell]} \Pj(R_1):\qquad&  d\theta(\partial_{\alpha_1}) = (x^2 + 2\epsilon\, xy, xy^2 + \epsilon\, y^3, xyz + \epsilon\, y^2 z, xz^2 + \epsilon\, yz^2,y^2z^2,yz^3,z^4, \\[-4pt]
			 &\hspace{1.9cm} v + b_4\epsilon\, y^2 + b_5 \epsilon\, yz) \\ 
			&d\theta(\partial_{\alpha_2}) = (x^2 + 2\epsilon\, xz, xy^2 + \epsilon\, y^2 z, xyz + \epsilon\, y z^2, xz^2 + \epsilon\, z^3,y^2z^2,yz^3,z^4,\\[-4pt]
			 &\hspace{1.9cm}v+ b_4 \epsilon\, yz + b_5 \epsilon\, z^2)\\ 
			\mathsf{T}_{[v]} \Delta_\ell:\qquad& d\theta(\partial_{\beta_1}) = (x^2, xy^2, xyz, xz^2, y^2z^2, yz^3, z^4, v + \epsilon\, y^2z) \\ 
			&d\theta(\partial_{\beta_2})= (x^2, xy^2, xyz, xz^2, y^2z^2, yz^3, z^4, v + \epsilon\, yz^2) \\
			&d\theta(\partial_{\beta_3})= (x^2, xy^2, xyz, xz^2, y^2z^2, yz^3, z^4, v + \epsilon\, z^3) \\
			&d\theta(\partial_{\beta_4}) =(x^2, xy^2, xyz, xz^2, y^2z^2, yz^3, z^4, v + \epsilon\, xy) \\
			&d\theta(\partial_{\beta_5})= (x^2, xy^2, xyz, xz^2, y^2z^2, yz^3, z^4, v + \epsilon\, xz) \\
			\mathsf{T}_0 \mathbb A^3:\qquad&d\theta(\partial_{\tau_x})= (x^2 + 2\epsilon\, x, xy^2 + \epsilon\, y^2, xyz + \epsilon\, yz, xz^2 + \epsilon\, z^2, y^2z^2, yz^3, z^4,\\[-4pt]
			 &\hspace{1.9cm} v + b_4 \epsilon\, y + b_5 \epsilon\, z) \\ 
			&d\theta(\partial_{\tau_y})= (x^2, xy^2 + 2\epsilon\, xy, xyz + \epsilon\, xz, xz^2, y^2z^2 + 2\epsilon\, yz^2, yz^3 + \epsilon\, z^3, z^4,\\[-4pt]
			&\hspace{1.9cm} v + 3\epsilon\, y^2 + 2b_1\epsilon\, yz + b_2 \epsilon\, z^2 + b_4 \epsilon\, x) \\
			&d\theta(\partial_{\tau_z})= (x^2, xy^2, xyz +\epsilon\, xy, xz^2 + 2\epsilon\, xz, y^2z^2 + 2\epsilon\, y^2z, yz^3 + 3\epsilon\, yz^2,\\[-4pt]
			&\hspace{1.9cm}  z^4 + 4\epsilon\, z^3,v + b_1\epsilon\, y^2 + 2b_2 \epsilon\, yz + 3b_3\epsilon\, z^2 + b_5 \epsilon\, x),
		\end{split}
		\]
		thus $\theta$ is a closed immersion of the irreducible and smooth variety $V$ of dimension 10.
	\end{proof}

	\appendix
	\section{Computational methods: marked bases}\label{APPENDICE}
	In this section we give further details about the computational content of the paper. We also provide ancillary files based on the {\em M2} Package \href{www.paololella.it/parity-conjecture/CounterexampleParityConjectureAuxiliaryTools.m2}{\small\tt CounterexampleParityConjectureAuxiliaryTools.m2}.
	
	In the polynomial ring $\C[x_1,\ldots,x_n]$, we assume that the variables are ordered as $x_1 < x_2 < \ldots < x_n$. Given a monomial $x^{\underline{\alpha}} \neq 1 \in \C[x_1,\ldots,x_n]$, we define the \textit{minimum} and the \textit{maximum} of a $x^{\underline{\alpha}} \neq 1$ as the smallest and largest variable dividing $x^{\underline{\alpha}}$ and we define its \textit{Pommaret cone} as the set of monomials
	\[
	\mathcal{C}(x^{\underline{\alpha}}) = \left\{x^{\underline{\alpha}}\right\} \cup \left\{x^{\underline{\alpha}}\ x^{\underline{\beta}}\ \middle\vert\ \max x^{\underline{\beta}} \leqslant \min x^{\underline{\alpha}} \right\}.
	\]

	The monomials in the complementary set of an order ideal $\mathcal{N}$ generate a monomial ideal in $\C[x_1,\ldots,x_n]$ that we denote by $J_{\mathcal{N}}$. If $\mathcal{N}$ is finite, the ideal $J_{\mathcal{N}}$ is always \emph{quasi-stable} \cite[Corollary 2.3]{BertoneCioffiRoggero}, i.e.~$J_{\mathcal{N}}$ has a finite set of generators $\mathcal{P}_{J_{\mathcal{N}}} = \{x^{\underline{\alpha}_1},\ldots,x^{\underline{\alpha}_k}\}$ called \textit{Pommaret basis} of $J_{\mathcal{N}}$ such that the set of monomials in $J_{\mathcal{N}}$ decomposes as the union of the Pommaret cones
	\[
	\mathcal{C}(x^{\underline{\alpha}_1}) \sqcup\ \cdots\ \sqcup \mathcal{C}(x^{\underline{\alpha}_k}).
	\]
	of the elements of the Pommaret basis.
	\begin{definition}
		Let $J \in \C[x_1,\ldots,x_n]$ be a monomial ideal of finite colength, $\mathcal{P}_J$ be its  Pommaret basis and $\mathcal{N}$ be the finite order ideal whose image in the quotient is a monomial basis. A (monic) \emph{$J$-marked set} is a set of polynomials
		\[
		F := \left\{ f_{\underline{\alpha}} := x^{\underline{\alpha}} + \sum_{x^{\underline{\gamma}}\in\mathcal{N}} c_{\underline{\alpha},\underline{\gamma}} x^{\underline{\gamma}} \ \middle\vert\ x^{\underline{\alpha}} \in \mathcal{P}_J,\   c_{\underline{\alpha},\underline{\gamma}} \in \C\right\}.
		\]
		
		A $J$-marked set $F$ is said \emph{$J$-marked basis} if 
		\[
		\C[x_1,\ldots,x_n] = (F) \oplus \langle \mathcal{N}\rangle
		\]
		as vector spaces.
	\end{definition}
	
	The notion of marked basis generalises the notion of reduced Gr\"obner basis. In particular, from the definition it follows that $\mathcal{N}$ is a basis of the quotient algebra $\C[x_1,\ldots,x_n]/(F)$. Hence, every ideal $I \subset \C[x_1,\ldots,x_n]$ generated by a $J$-marked set defines a zero-dimensional scheme of length $\vert \mathcal{N}\vert$.
	
	With the notion of marked bases, we can rephrase the definition of marked family functor as follows
	\[
	\underline{\Mf}_{\mathcal{N}}(A) = \left\{   (F) \subset A[x_1,\ldots,x_n]\ \middle\vert\ F \text{~is a $J_{\mathcal{N}}$-marked basis}\right\}.
	\]
	This formulation has a crucial role because
	there is an effective (and algorithmic) criterion to determine whether a $J$-marked set is a $J$-marked basis. Given a $J$-marked set $F = \{f_{\underline{\alpha}}\ \vert\ x^{\underline{\alpha}} \in \mathcal{P}_J\}$, we denote by $\xrightarrow{F}$ the transitive closure of the relation $g \xrightarrow{} g - (c x^{\underline{\beta}}) f_{\underline{\alpha}}$, where $x^{\underline{\beta}} x^{\underline{\alpha}}$ is a monomial appearing in $g$ with coefficient $c\neq 0$ and $x^{\underline{\beta}} \in \mathcal{C}(x^{\underline{\alpha}})$. This reduction procedure is noetherian \cite[Theorem 5.9]{CeriaMoraRoggero} \cite[Proposition 4.3]{BertoneCioffiRoggero}, that is starting from any polynomial $g$ we obtain with a finite number of steps a polynomial $h \in \langle \mathcal{N}\rangle$. We write $g \xrightarrow{F}_{\ast} h$ to denote the beginning and the end of the reduction process and we call $h  \in \langle \mathcal{N}\rangle$ the $J$-normal form of $g$ with respect to $F$.
	
	\begin{theorem}[{\cite[Proposition 5.6]{BertoneCioffiRoggero}}]\label{thm:effective} 
		Given\hfill a\hfill zero-dimensional\hfill quasi-stable\hfill ideal\hfill $J$\hfill in\\ $\C[x_1,\ldots,x_n]$, a $J$-marked set $F= \{f_{\underline{\alpha}}\ \vert\ x^{\underline{\alpha}} \in \mathcal{P}_J\}$ is a $J$-marked basis if and only if,
		\[
		x_i f_{\underline{\alpha}} \xrightarrow{F}_\ast 0,\qquad\forall\ f_{\underline{\alpha}} \in F,\quad \forall\ x_i > \min x^{\underline{\alpha}}.
		\]
	\end{theorem}
	
	\begin{example} 
		In the polynomial ring $\C[x,y]$ (with $x>y$), consider the monomial ideal $J = (x^2,y)$. The monomial basis $\mathcal{N}$ of $\C[x,y]/J$ is the pair $\{x,1\}$ and the Pommaret basis of $J$ is $\mathcal{P}_J = \{x^2,xy,y\}$. A $J$-marked set $F$ is a triple of polynomials 
		\[
		f_{x^2} = x^2 + c_1 x + c_2,\qquad f_{xy} = xy + c_3 x + c_4,\qquad f_{y} = y + c_5 x + c_6.
		\]
		In order to have $J$-marked basis, the $J$-normal forms of $xf_{xy}$ and $xf_y$ have to be zero. One has
		\[
		\begin{split}  
			xf_{xy} = x^2y + c_3x^2 + c_4x &\quad\xrightarrow{-yf_{x^2}}\quad c_3x^2 - c_1xy + c_4x - c_2 y \\[-6pt]
    &\quad\xrightarrow{-c_3 f_{x^2}}\quad -c_1 xy + (c_4-c_1c_3)x - c_2 y - c_2c_3 \\[-6pt]
			&\quad\xrightarrow{+c_1 f_{xy}}\quad  c_4 x - c_2 y + c_1c_4 - c_2c_3 \\[-6pt]
   &\quad\xrightarrow{+c_2 f_{y}}\quad (c_4+c_2c_5)x  + (c_1c_4 - c_2c_3 + c_2c_6), \\
			xf_{y} = xy + c_5x^2 + c_6x &\quad\xrightarrow{-f_{xy}}\quad c_5x^2 - (c_3-c_6)x - c_4 \\[-6pt]
   &\quad\xrightarrow{-c_5f_{x^2}}\quad -(c_3-c_6+c_1c_5)x - (c_4 + c_2c_5).
		\end{split}
		\]
		Hence, $F$ is a marked basis if and only if
		\[
		c_3-c_6+c_1c_5 = 0\quad\text{and}\quad c_4 + c_2c_5 = 0
		\]
		(the third equation is redundant). These equations define the open subset of $\Mf_{\mathcal{N}}\subset \Hilb^2 \A^2$ and guarantee that the generator $f_{xy}$ is redundant as expected. In fact, $f_{xy} = xy + (c_6-c_1c_5)x - c_2c_5 = xf_y - c_5f_{x^2}$.
		
		If we want to describe the tangent space $\mathsf{T}_{\left[(F)\right]} \Hilb^2 \A^2 \simeq \mathsf T_{[(F)]}\Mf_{\mathcal{N}} \simeq \Hom_{(F)}(D, \Mf_{\mathcal{N}})$, we can compute the set of flat families $\underline{\Mf}_{\mathcal{N}}(D)$ with the unique closed point corresponding to $(F)$. We start with the marked set $\tilde{F}$ consisting of the following polynomials in $\C[\epsilon][x,y]$
		\[
		\tilde{f}_{x^2} = x^2 + (c_1+T_1\epsilon) x + (c_2+T_2\epsilon),\ \tilde{f}_{xy} = xy + (c_3 +T_3\epsilon) x + (c_4+T_4\epsilon),\ \tilde{f}_{y} = y + (c_5+T_5\epsilon) x + (c_6+T_6\epsilon),
		\]
		where $T_1,\ldots,T_6$ are complex parameters, and we impose the flatness via Theorem \ref{thm:effective}. Assuming that $F$ is a $J$-marked basis, the same holds for $\tilde{F}$ if
		\[
		\begin{split}
			x\tilde{f}_{xy}\quad &\xrightarrow{\tilde F}_\ast\quad \big(c_5T_2 + T_4 + c_2T_5\big)\epsilon\, x + \big(c_4T_1 - (c_3-c_6)T_2 - c_2T_3 + c_1 T_4 +c_2T_6\big)\epsilon= 0,\\
			x\tilde{f}_{y}\quad &\xrightarrow{\tilde F}_\ast\quad \big(c_5T_1 + T_3 + c_1T_5 - T_6\big)\epsilon\, x + \big(c_5T_2  + T_4 + c_2 T_5\big)\epsilon= 0.
		\end{split}
		\]
		By solving the linear system
		\[
		\left[
		\begin{array}{cccccc}
			0 & c_5 & 0 & 1 & c_2 & 0 \\ c_5 & 0 & 1 & 0 & c_1 & -1 
		\end{array}
		\right]
		\begin{bmatrix}
			T_1 \\ T_2 \\ T_3 \\ T_4 \\ T_5 \\ T_6
		\end{bmatrix}
		=
		\left[
		\begin{array}{c}
			0 \\ 0
		\end{array}
		\right],
		\]
		(the third equation is again redundant) one obtains a complete description of first order deformations of the ideal generated by $F$.
	\end{example}

	\subsection{Computational details of the proof of Theorem \ref{thmA}}\label{subsec:counterexample}

	The detailed computation of the tangent space of ideals of type \eqref{eq:generic-counterexample} is available in the \emph{Macaulay2} ancillary file \href{\linkProofCounterExample}{\tt\small proof\_counterexamples.m2}. Here, we report a summary of the procedure and the results.
	
	Assume that in $R=\C[x,y,z]$, $x>y>z$ and consider the quasi-stable ideal 
	\[
	J = \big(x + (y,z)^2\big)^2 + \big(y^3) = \left(x^2,xy^2,y^3,xyz,xz^2,y^2z^2,yz^3,z^4\right).
	\]
	The minimal set of generators of $J$ is its Pommaret basis $\mathcal{P}_J$. An ideal $I$ of the form  \eqref{eq:generic-counterexample} is generated by the $J$-marked basis\footnote{Notice that the polynomial $f_{y^3}$ is called $v$ in Theorem \ref{thmA}.}
	\[
	\begin{split}
		&f_{x^2} = x^2,\quad f_{xy^2} = xy^2,\quad f_{y^3} = y^3 + b_1y^2z+b_2yz^2+b_3z^3+b_4xy+b_5xz,\\
		&f_{xyz} = xyz,\quad f_{xz^2} = xz^2,\quad f_{y^2z^2}=y^2z^2,\quad f_{yz^3}=yz^3,\quad f_{z^4}=z^4.
	\end{split}
	\]
	Thus, $[I] \in \Mf_{\mathcal{N}}$ with $\mathcal{N}$ the order ideal
	\[
	\mathcal{N} = \left\{ 1, x, y, z, xy, y^2, xz, yz, z^2, y^2z,yz^2,z^3\right\}.
	\]
	We compute $\mathsf T_{[I]}\Hilb^{12}\A^3 = \mathsf T_{[I]}\Mf_{\mathcal{N}_J}$ as the set of morphisms $\text{Hom}_I(D,\Mf_{\mathcal{N}})$ mapping the unique closed point of $D$ to $[I]$. This is equivalent to classify all $J$-marked sets $\widetilde{F}$ consisting in polynomials in $\C[\epsilon][x,y,z]$ of the form
	\[
	\begin{split}
		\tilde{f}_{\bullet} = f_{\bullet} &{}+ T_{\bullet,1}\epsilon\, y^2z +  T_{\bullet,2}\epsilon\, yz^2 + T_{\bullet,3}\epsilon\, z^3 + T_{\bullet,4}\epsilon\, xy + T_{\bullet,5}\epsilon\, y^2 + T_{\bullet,6}\epsilon\, xz + {} \\ &{}+ T_{\bullet,7}\epsilon\, yz + T_{\bullet,8}\epsilon\, z^2 + T_{\bullet,9}\epsilon\, x + T_{\bullet,10}\epsilon\, y + T_{\bullet,11}\epsilon\, z + T_{\bullet,12}\epsilon,\qquad \forall\ f_{\bullet} \text{~generator of~}I
	\end{split}
	\]
	that are $J$-marked bases. 
	
	Imposing that
	\[
	\begin{split}
		& x\tilde{f}_{xy^2} \xrightarrow{F}_\ast 0,\quad  x\tilde{f}_{y^3} \xrightarrow{F}_\ast 0,\quad  x\tilde{f}_{xyz} \xrightarrow{F}_\ast 0,\quad  y\tilde{f}_{xyz} \xrightarrow{F}_\ast 0,\quad
		x\tilde{f}_{xz^2} \xrightarrow{F}_\ast 0,\quad  y\tilde{f}_{xz^2} \xrightarrow{F}_\ast 0,\\  
		&x\tilde{f}_{y^2z^2} \xrightarrow{F}_\ast 0,\quad y\tilde{f}_{y^2z^2} \xrightarrow{F}_\ast 0,\quad  x\tilde{f}_{yz^3} \xrightarrow{F}_\ast 0,\quad  y\tilde{f}_{yz^3} \xrightarrow{F}_\ast 0,\quad x\tilde{f}_{z^4} \xrightarrow{F}_\ast 0,\quad  y\tilde{f}_{z^4} \xrightarrow{F}_\ast 0
	\end{split}
	\]
	gives rise to a linear system of 98 equations in $96$ variables. Looking at the minors of the system matrix $S$, we get that $42 \leqslant \text{rk}(S) \leqslant 51$ and
	\[
	\text{rk}(S) < 43,\ \text{if~}b_4=b_5 = 0,\qquad \text{rk}(S) < 51,\ \text{if~}\mathfrak B=0,
	\]
	where $\mathfrak B = b_3b_4^3 - b_2b_4^2b_5 + b_1b_4b_5^2 - b_5^3$ is the same polynomial which appears in the proof of Theorem \ref{thmA}.
	Hence,
	\begin{equation}\label{eq:tableTangentSpaceDim}
	\dim_{\C} \mathsf{T}_{[I]} \Hilb^{12} \A^3 = \begin{cases}
		54,&\text{if~}b_4=b_5 = 0\\
		48,&\text{if~}\mathfrak B =0 \text{~and~} \vert b_4 \vert + \vert b_5 \vert\neq 0\\
		45,&\text{if~}\mathfrak B \neq 0.	
	\end{cases}
	\end{equation}
	
	\begin{remark}
		Under the condition $b_4 = b_5 = 0$, the ideal $I$ is homogeneous. Hence,
		the parity conjecture holds true for $I$ \cite{SAMMARTANO}.
		
		For ideals with tangent space of dimension $48$, we notice that $b_4 \neq 0$ (in fact $b_4 = b_3b_4^3 - b_2b_4^2b_5 + b_1b_4b_5^2 - b_5^3 = 0$ implies $b_5 = 0$).  Thus, one has that
		\[
		b_3 = b_2\frac{b_5}{b_4} - b_1 \left(\frac{b_5}{b_4}\right)^2 + \left(\frac{b_5}{b_4}\right)^3
		\]
		and the polynomial $f_{y^3}$ becomes reducible
		\[
		\begin{split}
			&y^3 + b_1 y^2 z + b_2 y z^2 + \left(b_2\frac{b_5}{b_4} - b_1 \left(\frac{b_5}{b_4}\right)^2 + \left(\frac{b_5}{b_4}\right)^3\right) z^3 + b_4 x y + b_5 xz = {}\\
			&\qquad\qquad {} = \frac{1}{b_4^3} \left(b_4y+b_5 z\right) \big( b_4^2 y^2 + b_4(b_1b_4-b_5) yz + (b_2b_4^2-b_1b_4b_5+b_5^2)z^2 + b_4^3 x\big).
		\end{split}
		\]
	\end{remark}

 \begin{remark}\label{rem:dim char}
 The classification \eqref{eq:tableTangentSpaceDim} of the dimension of the tangent space $\mathsf{T}_{[I]} \Hilb^{12} \A^3$ holds for every characteristic other than 2.
 If the characteristic is 2, the rank of the matrix $S$ is at most $50$ and for $\mathfrak{B}\neq 0$ the tangent space has dimension 46.
 \end{remark}

\subsection{Computational details of the proof of Proposition \ref{prop:smoothability}}\label{subsec:smoothability}
	
The detailed computation of the smoothing deformation is available in the \emph{M2} ancillary file \href{\linkProofSmoothability}{\tt\small proof\_smoothability.m2}. 
	
The ideal \eqref{eq:smoothingDef} for $t \neq 0$ is the intersection of the ideals
\[
\begin{split}
	P_1 ={}&  \big( (x - b_4t^2)^2, y^2,\ (b_3b_4 - b_2b_5)yz - b_3b_5\, z^2 - b_5^2 (x - b_4t^2),\ b_3 yz^2 + b_5y(x - b_4t^2), \\
	&\quad b_2 yz^2 + b_3z^3 + (b_4y+b_5z)(x - b_4t^2),\ (x - b_4t^2)z^2,\ (x - b_4t^2) yz,\ yz^3,\ z^4 \big),\\
	P_2 ={}&  \Bigg( 
		x - \tfrac{b_4^3(b_3b_4-b_2b_5)}{\mathfrak B}t^2,\ b_4y+b_5z,\ \left(x - \tfrac{b_4^3(b_3b_4-b_2b_5)}{\mathfrak B}t^2\right)\Big(b_5^2y + \big(b_4(b_3b_4-b_2b_5)+b_1b_5^2\big)z\Big) +{}\\ &\quad \tfrac{b_4^3b_5^2(b_3b_4-b_2b_5)}{\mathfrak B}t^2\, y + \tfrac{b_4^2b_5^3(b_3b_4-b_2b_5)}{\mathfrak B}t^2\, z, y^2, yz, z^2 \Bigg),\\
	P_3 ={}&  \big( x,\ 2b_4( y - b_4t) + (b_1b_4-b_5)z,\ (2y+b_1z)(y-b_4t) + (b_1b_4-b_5)t\, z, (y-b_4t)z, z^2 \big),\\
	P_4 = {}& \big(x,y+b_4t,z\big),
\end{split}
\]
of length $7,2,2,1$ and supported at the points 
\[
(b_4t^2,0,0),\qquad\left(\tfrac{b_4^3(b_3b_4-b_2b_5)}{\mathfrak{B}}t^2,0,0\right),\qquad(0,b_4t,0),\qquad(0,-b_4t,0).
\]
	
Figure \ref{fig:collision} illustrates the deformation for $b_1 = 0$, $b_2 = -\frac{1}{2}$, $b_3 = 0$, $b_4 = 1$ and $b_5 = 1$.
In the figure any arrangement of boxes represents a chosen order ideal describing a monomial basis for $R/P_i$ for $i=1, \ldots, 4$ and $R/I$.
	
We remark that ideals $I$ corresponding to points in the locus $\mathfrak{B}=0$ still have a first order deformation involving the generator $xyz$ perturbed by $\epsilon y^2z$. But if $\mathfrak{B}=0$ and $b_4 \neq 0$, this deformation is obstructed. And, if $b_4=b_5 =0$, then the first order deformation lifts to the family
\[
\left(x^2,\ xy^2,\ xyz + \boldsymbol{t} y^2z,\ xz^2,\ y^2z^2,\ yz^3,\ z^4,\  y^3 + b_1y^2z+b_2yz^2+b_3z^3\right)
\]
whose generic fiber is irreducible.

\newcommand{\etalchar}[1]{$^{#1}$}
\providecommand{\bysame}{\leavevmode\hbox to3em{\hrulefill}\thinspace}
\providecommand{\MR}{\relax\ifhmode\unskip\space\fi MR }
\providecommand{\MRhref}[2]{%
  \href{http://www.ams.org/mathscinet-getitem?mr=#1}{#2}
}


\begin{thebibliography}{MNOP06}

\bibitem[BC22]{BertoneCioffi}
Cristina Bertone and Francesca Cioffi, \emph{The close relation between border
  and {P}ommaret marked bases}, Collect. Math. \textbf{73} (2022), no.~2,
  181--201.

\bibitem[BCLR13]{BertoneCioffiLellaRoggero}
Cristina Bertone, Francesca Cioffi, Paolo Lella, and Margherita Roggero,
  \emph{Upgraded methods for the effective computation of marked schemes on a
  strongly stable ideal}, J. Symbolic Comput. \textbf{50} (2013), 263--290.

\bibitem[BCR17]{BertoneCioffiRoggero}
Cristina Bertone, Francesca Cioffi, and Margherita Roggero, \emph{Macaulay-like
  marked bases}, J. Algebra Appl. \textbf{16} (2017), no.~5, 1750100, 36.

\bibitem[Beh09]{Beh}
Kai Behrend, \emph{Donaldson--Thomas type invariants via microlocal
  geometry}, Ann. of Math. \textbf{2} (2009), no.~170, 1307--1338.

\bibitem[BLR13]{BertoneLellaRoggero}
Cristina Bertone, Paolo Lella, and Margherita Roggero, \emph{A {B}orel open
  cover of the {H}ilbert scheme}, J. Symbolic Comput. \textbf{53} (2013),
  119--135.

\bibitem[CEVV09]{8POINTS}
Dustin~A. Cartwright, Daniel Erman, Mauricio Velasco, and Bianca Viray,
  \emph{Hilbert schemes of 8 points}, Algebra Number Theory \textbf{3} (2009),
  no.~7, 763--795.

\bibitem[CMR15]{CeriaMoraRoggero}
Michela Ceria, Teo Mora, and Margherita Roggero, \emph{Term-ordering free
  involutive bases}, J. Symbolic Comput. \textbf{68} (2015), no.~part 2,
  87--108.

\bibitem[CR11]{CioffiRoggero}
Francesca Cioffi and Margherita Roggero, \emph{Flat families by strongly stable
  ideals and a generalization of {G}r{\"o}bner bases}, J. Symbolic Comput.
  \textbf{46} (2011), no.~9, 1070--1084.

\bibitem[DJNT17]{JOACHIM}
Theodosios Douvropoulos, Joachim Jelisiejew, Bernt Ivar~Utst{\o}l N{\o}dland,
  and Zach Teitler, \emph{The {H}ilbert scheme of 11 points in {$\mathbb A^3$}
  is irreducible}, Combinatorial algebraic geometry, Fields Inst. Commun.,
  vol.~80, Fields Inst. Res. Math. Sci., Toronto, ON, 2017, pp.~321--352.

\bibitem[Eis95]{EISENBUD}
David Eisenbud, \emph{Commutative algebra}, Graduate Texts in Mathematics, vol.
  150, Springer-Verlag, New York, 1995, With a view toward algebraic geometry.

\bibitem[FGI{\etalchar{+}}05]{FGAexplained}
Barbara Fantechi, Lothar G\"{o}ttsche, Luc Illusie, Steven~L. Kleiman, Nitin
  Nitsure, and Angelo Vistoli, \emph{Fundamental algebraic geometry},
  Mathematical Surveys and Monographs, vol. 123, American Mathematical Society,
  Providence, RI, 2005, Grothendieck's FGA explained.

\bibitem[Fog68]{FOGARTY}
John Fogarty, \emph{Algebraic families on an algebraic surface}, Amer. J. Math.
  \textbf{90} (1968), 511--521.

\bibitem[GR23]{GRICOLFI}
Michele Graffeo and Andrea~T. Ricolfi, \emph{On the Behrend function and the
  blowup of some fat points}, Advances in Mathematics \textbf{415} (2023),
  108896.

\bibitem[Gro62]{FGA}
Alexander Grothendieck, \emph{Fondements de la g\'{e}om\'{e}trie
  alg\'{e}brique. {E}xtraits du {S}\'{e}minaire {B}ourbaki, 1957--1962.},
  Secr\'{e}tariat math\'{e}matique, Paris, 1962.

\bibitem[GS]{M2}
Daniel~R. Grayson and Michael~E. Stillman, \emph{Macaulay2, a software system
  for research in algebraic geometry}, Available at
  \url{http://www.math.uiuc.edu/Macaulay2/}.

\bibitem[HJ18]{10points}
Abdelmoubine~A. Henni and Marcos Jardim, \emph{Commuting matrices and the
  hilbert scheme of points on affine spaces}, Advances in Geometry \textbf{18}
  (2018), no.~4, 467--482.

\bibitem[Iar84]{IARRO}
Anthony Iarrobino, \emph{Compressed algebras: {A}rtin algebras having given
  socle degrees and maximal length}, Trans. Amer. Math. Soc. \textbf{285}
  (1984), no.~1, 337--378.

\bibitem[IE78]{Iarrob}
Anthony Iarrobino and Jacques Emsalem, \emph{Some zero-dimensional generic
  singularities; finite algebras having small tangent space}, Compositio Math.
  \textbf{36} (1978), no.~2, 145--188.

\bibitem[Jel]{JoachimQuestions}
Joachim Jelisiejew, \emph{Open problems in deformations of Artinian
  algebras, Hilbert schemes and around}, {to appear in a Contemporary
  Mathematics (AMS) volume}.

\bibitem[Jel20]{joachim-pathologies}
\bysame, \emph{Pathologies on the {H}ilbert scheme of points}, Invent. Math.
  \textbf{220} (2020), no.~2, 581--610.

\bibitem[JKS23]{behrendnonconstant}
Joachim Jelisiejew, Martijn Kool, and Reinier~F. Schmiermann, \emph{Behrend's
  function is not constant on $\mathrm{Hilb}^n(\mathbb{C}^3)$}, arXiv e-prints \href{https://browse.arxiv.org/abs/2311.05408}{\tt arXiv:2311.05408} (2023).

\bibitem[LR16]{LellaRoggero}
Paolo Lella and Margherita Roggero, \emph{On the functoriality of marked
  families}, J. Commut. Algebra \textbf{8} (2016), no.~3, 367--410.

\bibitem[Maz80]{MAZZOLA}
Guerino Mazzola, \emph{Generic finite schemes and {H}ochschild cocycles},
  Comment. Math. Helv. \textbf{55} (1980), no.~2, 267--293.

\bibitem[MNOP06]{PANDHARIPANDE}
Davesh Maulik, Nikita Nekrasov, Andrei Okounkov, and Rahul Pandharipande,
  \emph{Gromov-{W}itten theory and {D}onaldson-{T}homas theory. {I}}, Compos.
  Math. \textbf{142} (2006), no.~5, 1263--1285.

\bibitem[Pan22]{CONGETTURA}
Rahul Pandharipande, \emph{\href{https://people.math.ethz.ch/~rahul/DMVF.pdf}{
  A tour of the geometry of points in affine space}}, DMV Jahrestagung lecture,
  2022.

\bibitem[Ric24]{SignAnn}
Andrea~T. Ricolfi, \emph{A sign that used to annoy me, and still does}, J.
  Geom. Phys. \textbf{195} (2024), Paper No. 105032, 7.

\bibitem[RS23]{SAMMARTANO}
Ritvik Ramkumar and Alessio Sammartano, \emph{{On the parity conjecture for
  Hilbert schemes of points on threefolds}}, Annali della Scuola Normale
  Superiori di Pisa, Classe di Scienze, 2025, Vol. XXVI, Issue 3 (2023).

\bibitem[Sza21]{szachniewicz2021nonreducedness}
Michał Szachniewicz, \emph{Non-reducedness of the Hilbert schemes of few
  points}, arXiv e-prints \href{https://browse.arxiv.org/abs/2109.11805}{\tt arXiv:2109.11805} (2021).

\bibitem[Vak06]{Vakil}
Ravi Vakil, \emph{Murphy's law in algebraic geometry: badly-behaved deformation
  spaces}, Invent. Math. \textbf{164} (2006), no.~3, 569--590.

\bibitem[\v{S}12]{Klemen}
Klemen \v{S}ivic, \emph{On varieties of commuting triples {III}}, Linear
  Algebra Appl. \textbf{437} (2012), no.~2, 393--460.

\end{thebibliography}
\end{document}